\documentclass[preprint,12pt]{elsarticle}        
 
\usepackage{amsfonts,amsmath,amssymb,amsthm}

\theoremstyle{plain}
\newtheorem{theorem}{Theorem}[section]

\theoremstyle{definition}
\newtheorem{question}[theorem]{Question}

\theoremstyle{remark}
\newtheorem*{remark}{Remark}

\newcommand{\id}{{\mathbf 1}}
\newcommand{\CC}{{\mathbb C}}
\newcommand{\RR}{{\mathbb R}}
\newcommand{\Ss}{{\mathbb S}}
\newcommand{\TT}{{\mathbb T}}

\let\Im\undefined
\DeclareMathOperator{\Im}{Im}
\DeclareMathOperator{\Inv}{Inv}
\DeclareMathOperator{\Exp}{Exp}

\journal{journal}

\begin{document}

\begin{frontmatter}

\title{Non-commutativity of the exponential spectrum}

\author[els]{Hubert Klaja}
\ead{hubert.klaja@gmail.com}
\author[els]{Thomas Ransford\corref{cor1}\fnref{fn1}}
\ead{thomas.ransford@mat.ulaval.ca}

\address[els]{D\'epartement de math\'ematiques et de statistique, 
Universit\'e Laval, \\ 1045, avenue de la M\'edecine, 
Qu\'ebec (Qu\'ebec), Canada G1V 0A6}
\cortext[cor1]{Corresponding author}
\fntext[fn1]{Supported by grants from NSERC and
the Canada Research Chairs program}

\begin{abstract}
In a Banach algebra, the spectrum satisfies 
$\sigma(ab)\setminus\{0\} = \sigma(ba)\setminus\{0\}$
for each pair of elements $a,b$.
We show that this is no longer true for the exponential spectrum, 
thereby solving a problem open since 1992. 
Our proof depends on the fact that the homotopy group 
$\pi_4(GL_2(\CC))$ is non-trivial.
\end{abstract}

\begin{keyword}
Banach algebra \sep exponential spectrum \sep homotopy \sep Hopf map
\MSC[2010]  46H05 \sep 55P15
\end{keyword}

\end{frontmatter}

\section{Introduction}\label{S:introduction}

Let $A$ be a complex, unital Banach algebra.
We denote by $\Inv(A)$ the set of invertible elements of $A$.
According to a standard result \cite[Theorem~2.14]{Do98},
the connected component of $\Inv(A)$
containing the identity is equal to the set
\[
\Exp(A):=\{e^{a_1}e^{a_2}\cdots e^{a_n}: a_1,\dots,a_n\in A,~n\ge1\}.
\]
$\Exp(A)$ is a normal subgroup of $\Inv(A)$, 
and the quotient $\Inv(A)/\Exp(A)$ is called 
the \emph{index group} of $A$. 

The \emph{spectrum} $\sigma(a)$ and
\emph{exponential spectrum} $\epsilon(a)$ 
of an element $a\in A$ are defined by
\begin{align*}
\sigma(a) &:=\{\lambda\in\CC: \lambda-a\notin\Inv(A)\},\\
\epsilon(a) &:= \{\lambda\in\CC: \lambda-a\notin\Exp(A)\}.
\end{align*}
They are  compact sets, related by the inclusions
$\partial\epsilon(a)\subset\sigma(a)\subset\epsilon(a)$.
Thus $\epsilon(a)$ is obtained from $\sigma(a)$ 
by filling some (or possibly no) holes of $\sigma(a)$.

The exponential spectrum was introduced by Harte \cite{Ha76}
to exploit the fact that, if $\theta:A\to B$ is 
a surjective continuous homomorphism of Banach algebras,
then $\theta(\Exp(A))=\Exp(B)$
(whereas the inclusion 
$\theta(\Inv(A))\subset\Inv(B)$ may be strict).
The exponential spectrum was also used by Murphy \cite{Mu92}
to simplify proofs of certain results about Toeplitz operators.

Some properties of the spectrum are shared by the exponential spectrum, 
whilst others are not. Most of these are straightforward to analyze 
and were treated long ago. However, there is an interesting exception.
It is a useful result in several areas of mathematics that, 
for every pair of elements $a,b\in A$, we have
\begin{equation}\label{E:sigma(ab)}
\sigma(ab)\setminus\{0\}=\sigma(ba)\setminus\{0\}.
\end{equation}
This follows from the elementary algebraic fact that,
if $1-ab$ has inverse $u$, then $1-ba$ has inverse $1+bua$.
Murphy \cite{Mu92} raised the question 
as to whether this commutativity property for the spectrum 
is shared by the exponential spectrum. 
In other words, do we always have
\begin{equation}\label{E:epsilon(ab)}
\epsilon(ab)\setminus\{0\}=\epsilon(ba)\setminus\{0\}~?
\end{equation}
Until now, this has remained an open problem.

Murphy \cite{Mu92} proved that \eqref{E:epsilon(ab)} 
holds if either $a$ or $b$ belongs to $\overline{\Inv(A)}$,
the closure of the set of invertible elements.
A simple adaptation of his argument shows that, more generally, 
\eqref{E:epsilon(ab)} holds 
whenever $a$ or $b$ belongs to $\overline{Z(A)\Inv(A)}$,
where $Z(A)$ denotes the centre of $A$. We omit the details.

Thus, in seeking a counterexample to \eqref{E:epsilon(ab)},
we may restrict attention to Banach algebras $A$ satisfying 
\[
\Exp(A)\ne\Inv(A)
\quad\text{and}\quad 
\overline{Z(A)\Inv(A)}\ne A.
\]
This straightaway rules out many candidates.

Foremost among the algebras that do satisfy these requirements 
is the Calkin algebra, namely the quotient algebra $B(H)/K(H)$, 
where $H$ is a separable, infinite-dimensional Hilbert space, 
and $B(H)$ and $K(H)$ denote respectively 
the bounded operators and the compact operators on $H$. 
Thus the Calkin algebra would seem to be a  candidate 
for a counterexample to \eqref{E:epsilon(ab)}. 
However, as pointed out  by Murphy in \cite{Mu92}, 
the commutativity property \eqref{E:epsilon(ab)} 
holds in the Calkin algebra as well.

Indeed, it is a classical result of Atkinson that 
the invertible elements of $A:=B(H)/K(H)$ 
are the cosets of Fredholm operators, 
and that the Fredholm index function defines an isomorphism 
of $\Inv(A)/\Exp(A)$ onto the integers 
\cite[Theorems~5.15 and~5.35]{Do98}.
In particular, $\Exp(A)$ consists precisely of 
the cosets of  operators of Fredholm index zero. 
Since the Fredholm indices of operators 
$1-ab$ and $1-ba$ are always equal, 
it follows that \eqref{E:epsilon(ab)} holds in the Calkin algebra. 
An abstract version of this argument is presented in \cite{GR08}.

Our purpose in this paper is to give an example 
of a Banach algebra in which (\ref{E:epsilon(ab)}) does not hold.
We denote by $\Ss^{k}$ the unit sphere of $\RR^{k+1}$.
We view the sphere  as a subset of complex euclidean space 
via the identifications
\begin{align*}
\Ss^{2k+1}&:=\Bigl\{(z_0,\dots, z_k) \in \CC^{k+1}: \sum_{i=0}^k |z_i|^2=1\Bigr\},\\
\Ss^{2k}&:=\Bigl\{(z_0,\dots,z_k)\in\CC^{k+1} : \sum_{i=0}^k |z_i|^2=1,~\Im z_k=0\Bigr\}.
\end{align*}
Also, we write  $M_n(\CC)$ for the set of complex $n\times n$ matrices,
and $GL_n(\CC)$ for $\Inv(M_n(\CC))$.
Lastly, given topological spaces $X$ and $Y$,
we write $C(X,Y)$ for the set of continuous maps from $X$ into $Y$.
In what follows, we consider $C(\Ss^4,M_2(\CC))$,
which is a Banach algebra, indeed even a C*-algebra.

The main result of the paper is the following.

\begin{theorem}\label{T:main}
There exist $a,b\in C(\Ss^4,M_2(\CC))$ such that
\[
\epsilon(ab)\setminus\{0\} 
\neq \epsilon(ba)\setminus\{0\}.  
\]
\end{theorem}

The proof of Theorem~\ref{T:main} relies on the following result from topology. 

\begin{theorem}\label{T:algtop} 
The map $c:\Ss^4\to GL_2(\CC)$, defined by
\[
c(z_0,z_1,z_2) := 
\begin{pmatrix}
  1 & 0 \\ 
  0 & 1
\end{pmatrix}
- \frac{2}{(1+iz_2)^2}
  \begin{pmatrix}
  z_0 \overline{z}_0 & z_0 \overline{z}_1 \\ 
  z_1 \overline{z}_0 & z_1 \overline{z}_1
  \end{pmatrix},
\]
is not homotopic in $C(\Ss^4,GL_2(\CC))$ to a constant map.
\end{theorem}

This theorem is certainly known to specialists in topology.
For example, it is implicit in the discussion 
in \S24 of Steenrod's book \cite{St51}. 
However, we have been unable to find an explicit statement of the result 
in the form that we need, 
so we provide a proof for the convenience of the reader. 

The paper is organized as follows.
In Section~\ref{S:main}, we prove Theorem~\ref{T:main}
assuming that Theorem~\ref{T:algtop} holds. 
In Section~\ref{S:algtop}, we show how Theorem~\ref{T:algtop} 
can be deduced from standard results in algebraic topology. 
Finally, in Section~\ref{S:conclusion}, 
we make a few closing remarks and pose some questions.

\section{Proof of Theorem~\ref{T:main}}\label{S:main}

Let $X$ be a compact topological space and 
let $B$ be a unital Banach algebra.
Then $C(X,B)$ is a Banach algebra with respect to the sup norm, and 
\[
\Inv(C(X,B))=C(X,\Inv(B)).
\]
Also, $f\in\Exp(C(X,B))$ if and only if 
$f$ is homotopic to $\id$ in $C(X,\Inv(B))$,
where $\id$ denotes the identity element of $C(X,B)$.

For the rest of the section, 
we take $X:=\Ss^4 $ and $B:=M_2(\CC)$.

\begin{proof}[Proof of Theorem \ref{T:main}]
Let $A:=C(\Ss^4,M_2(\CC))$, and define  $a,b\in A$ by
\begin{align*}
a(z_0,z_1,z_2)&:= \frac{1}{1+iz_2} 
\begin{pmatrix}
z_0 &0\\
z_1 &0
\end{pmatrix},\\
b(z_0,z_1,z_2)&:= \frac{1}{1+iz_2}
\begin{pmatrix}
\overline{z}_0 &\overline{z}_1\\
0 &0
\end{pmatrix}.
\end{align*}

A calculation shows that
\[
(\id-2ab)(z_0,z_1,z_2)=
\begin{pmatrix}
  1 & 0 \\ 
  0 & 1
\end{pmatrix}
- \frac{2}{(1+iz_2)^2}
\begin{pmatrix}
  z_0 \overline{z}_0 & z_0 \overline{z}_1 \\ 
  z_1 \overline{z}_0 & z_1 \overline{z}_1
\end{pmatrix}.
 \]
In other words, we have $\id-2ab=c$, 
where $c$ is the mapping in Theorem~\ref{T:algtop}.
By that theorem, $c$ is not null-homotopic in $C(\Ss^4,GL_2(\CC))$, 
and so 
\begin{equation}\label{E:1-2ab}
\id-2ab\notin\Exp(A).
\end{equation}

On the other hand, 
recalling that $|z_0|^2+|z_1|^2+|z_2|^2=1$ and $\Im z_2=0$,
we have
\begin{align*}
(\id-2ba)(z_0,z_1,z_2)
&= \begin{pmatrix}
  1 & 0 \\ 
  0 & 1
\end{pmatrix}
- \frac{2(|z_0|^2+|z_1|^2)}{(1+iz_2)^2}
\begin{pmatrix}
  1 & 0 \\ 
  0 & 0
\end{pmatrix}\\
&= \begin{pmatrix}
  1 & 0 \\ 
  0 & 1
\end{pmatrix}
- \frac{2(1-z_2^2)}{(1+iz_2)^2}
\begin{pmatrix}
  1 & 0 \\ 
  0 & 0
\end{pmatrix}\\
&= \begin{pmatrix}
\phi(z_2) &0\\
0 &1
\end{pmatrix},
\end{align*}
where $\phi(z_2):=-((1-iz_2)/(1+iz_2))^2$.
Thus $\id-2ba$ factors as
\begin{align*}
\Ss^4 &&\to &&[-1,1] &&\to &&\TT &&\to &&GL_2(\CC)\\
(z_0,z_1,z_2) &&\mapsto &&z_2 &&\mapsto &&\phi(z_2) &&\mapsto 
&&\begin{pmatrix} \phi(z_2) &0\\ 0&1 \end{pmatrix},
\end{align*}
where $\TT$ denotes the unit circle.
Clearly $\phi$ is null-homotopic in $C([-1,1],\TT)$.
Therefore $\id-2ba$ is null-homotopic in $C(\Ss^4,GL_2(\CC))$, 
and so
\begin{equation}\label{E:1-2ba}
\id-2ba\in\Exp(A).
\end{equation}

Finally, combining \eqref{E:1-2ab} and \eqref{E:1-2ba},
we see that $1/2\in\epsilon(ab)$ but $1/2\notin\epsilon(ba)$.
This completes the proof.
\end{proof}

\begin{remark}
With $a$ and $b$ defined as above, we have
\[
\sigma(\id-2ba)=\phi([-1,1])\cup\{1\}=\TT.
\]
Thus $\sigma(ba)$ is  the circle $C$ with centre $1/2$ and radius $1/2$.
From \eqref{E:sigma(ab)}, it follows that $\sigma(ab)=C$ as well.
Using the facts 
$\partial\epsilon(\cdot)\subset\sigma(\cdot)\subset\epsilon(\cdot)$ 
and $1/2\in\epsilon(ab)$ and $1/2\notin\epsilon(ba)$, 
we deduce that
\[
\epsilon(ba)= C \quad\text{and}\quad \epsilon(ab)=D,
\]
where $D$ is the closed disk with centre $1/2$ and radius $1/2$,
obtained from $C$ by filling the hole.

\end{remark}

\section{Proof of Theorem~\ref{T:algtop}}\label{S:algtop}

In this section, 
we show how to deduce Theorem~\ref{T:algtop}
from standard results in topology.
Our treatment follows the same lines as 
the much broader discussion presented in Steenrod \cite[\S24]{St51}.

In complex coordinates, the Hopf map  $h:\Ss^3\to\Ss^2$ is given by
\[
h(z_0,z_1):=(-2z_0\overline{z}_1, ~|z_0|^2-|z_1|^2),
\]
and its suspension  $Eh:\Ss^4\to\Ss^3$ is  given by
\[
Eh(z_0,z_1,z_2):= 
\Bigl(\frac{-2z_0\overline{z}_1}{\sqrt{1-|z_2|^2}},~ 
\frac{|z_0|^2-|z_1|^2}{\sqrt{1-|z_2|^2}}+iz_2\Bigr).
\]
If we identify the equators of $\Ss^4$ and $\Ss^3$ 
with $\Ss^3$ and $\Ss^2$ respectively,
then on the equator of $\Ss^4$ we have $Eh=h$. 
Also $Eh$ maps the upper (respectively lower) hemisphere of $\Ss^4$ 
into the upper (respectively lower) hemisphere of $\Ss^3$.

Part~(i) of the following fundamental result is due to Hopf \cite{Ho31}, 
and part~(ii) follows from it, 
thanks to the suspension theorem of Freudenthal \cite{Fr38}. 
Proofs can also be found for example in \cite[\S III.5 and \S XI.2]{Hu59}.

\begin{theorem}\label{T:Eh} 
Let $h:\Ss^3\to\Ss^2$ and $Eh:\Ss^4\to\Ss^3$ 
be the Hopf map and its suspension, respectively.
Then:
\begin{enumerate}[\rm(i)]
\item $h$ is not null-homotopic in $C(\Ss^3,\Ss^2)$;
\item $Eh$ is not null-homotopic in $C(\Ss^4,\Ss^3)$.
\end{enumerate}
\end{theorem}

We now show how to deduce Theorem~\ref{T:algtop}.

\begin{proof}[Proof of Theorem \ref{T:algtop}]
Let $p:M_2(\CC)\to\CC^2$ denote the projection onto the second column,
and let $f:=pc/|pc|$, where $|\cdot|$ denotes the Euclidean norm. 
Then $f\in C(\Ss^4,\Ss^3)$.
We claim that $f$ is homotopic to $Eh$ in $C(\Ss^4,\Ss^3)$.
If so, then by Theorem~\ref{T:Eh} 
$f$ is not null-homotopic  in $C(\Ss^4,\Ss^3)$,
and it follows that 
$c$ cannot be null-homotopic in $C(\Ss^4,GL_2(\CC))$.

It remains to justify the claim. We have that
\[
pc(z_0,z_1,z_2)
=\Bigl(\frac{-2z_0\overline{z}_1}{(1+iz_2)^2},~1-\frac{2|z_1|^2}{(1+iz_2)^2}\Bigr).
\]
Considering the imaginary part of the second coordinate of $pc$,
we see that $f$ maps the upper (respectively lower) hemisphere of $\Ss^4$
into the upper (respectively lower) hemisphere of $\Ss^3$. 
Also, if $z_2=0$, then $|pc|=1$, so
\[
f(z_0,z_1,0)=(-2z_0\overline{z}_1,~1-2|z_1|^2)
=(-2z_0\overline{z}_1,~|z_0|^2-|z_1|^2).
\]
This shows that $f$ coincides with $Eh$ on the equator of $\Ss^4 $. 
Putting together these observations,
we deduce that $f(z_0,z_1,z_2)$ and $Eh(z_0,z_1,z_2)$ 
are never antipodal points on $\Ss^3$.
We can therefore define a homotopy 
between $f$ and $Eh$ in $C(\Ss^4,\Ss^3)$ via the formula 
\[
t\mapsto \frac{(1-t)f+t(Eh)}{|(1-t)f+t(Eh)|}.
\]
This justifies the claim, and completes the proof.
\end{proof}

\section{Concluding remarks and questions}\label{S:conclusion}

The commutativity property \eqref{E:epsilon(ab)} fails to hold
in $C(\Ss^{2n},M_n(\CC))$ for all $n\ge2$.
Indeed, defining
\[
a(z_0,\dots,z_n):=\frac{z\otimes e_1}{1+iz_n}
\quad\text{and}\quad
b(z_0,\dots,z_n):=\frac{e_1\otimes z}{1+iz_n},
\]
where $z:=(z_0,\dots,z_{n-1})$ and $e_1:=(1,0,\dots,0)\in\CC^n$,
we can modify the proof of Theorem~\ref{T:main} to show that
$1/2\in\epsilon(ab)$ but $1/2\notin\epsilon(ba)$.
We omit the details. 
Of course, when $n=1$, 
the algebra $C(\Ss^{2n},M_n(\CC))$ is commutative.

Banach algebras of the type $C(\Ss^k,M_n(\CC))$ 
were studied by Yuen \cite{Yu73},
who showed that, for certain values of $n$ and $k$, 
their index groups  are finite and non-trivial. 
In the same article she presented an example, due to Fadell, 
where the index group is finite and non-abelian.

Paulsen \cite{Pa82} gave another construction of Banach algebras  
with finite, non-trivial index groups.
In his examples, the invertible elements are dense, 
so, by one of the remarks in \S\ref{S:introduction},
the commutativity property \eqref{E:epsilon(ab)} always holds. 
This shows that the presence of 
non-trivial elements of finite order in the index group 
is not in itself sufficient to yield elements 
$a,b$ that violate \eqref{E:epsilon(ab)}. 
However, maybe the converse is true. 
We pose this as a question.

\begin{question}
Suppose that a Banach algebra $A$ contains elements $a,b$ 
such that $\epsilon(ab)\setminus\{0\}\ne\epsilon(ba)\setminus\{0\}$. 
Must the index group of $A$ contain 
a non-trivial element of finite order?
\end{question}

As remarked just before Theorem~\ref{T:main},
the Banach algebra $C(\Ss^4,M_2(\CC))$ is actually a C*-algebra.
However, it is not a von Neumann algebra.
Indeed \eqref{E:epsilon(ab)}
always holds in a von Neumann algebra, since the set
of invertible elements is connected
\cite[Proposition~5.29 and Corollary~5.30]{Do98}.
This leads us to speculate whether 
it is possible to find counterexamples to \eqref{E:epsilon(ab)}
within other standard classes of Banach algebras. 
In particular, the following question seems natural.

\begin{question}
Do there exist a Banach space $E$ 
and operators $S,T$ on $E$ such that 
$\epsilon(ST)\setminus\{0\}\ne\epsilon(TS)\setminus\{0\}$?
\end{question}

\end{document}